\documentclass[a4paper,11pt,reqno]{amsart}
\usepackage{amssymb,hyperref,stmaryrd,enumerate,pgf,tikz,color}
\usepackage[margin=2.75cm]{geometry}
\newtheorem{thm}{Theorem}[section]
\newtheorem{lem}[thm]{Lemma}
\newtheorem{prop}[thm]{Proposition}

\theoremstyle{remark}
\newtheorem{rem}[thm]{Remark}
\theoremstyle{definition}

\def\A{\mathcal{G}^*}
\def\G{\mathcal{G}}

\def\la{\langle}
\def\ra{\rangle}
\def\Z{{\mathbb Z}}
\DeclareMathOperator{\Aut}{Aut}
\DeclareMathOperator{\Cay}{\Gamma}
\DeclareMathOperator{\Dic}{Dic}
\DeclareMathOperator{\GF}{GF}

\DeclareMathOperator{\SL}{SL}
\DeclareMathOperator{\SU}{SU}
\DeclareMathOperator{\fix}{Fix}

\begin{document}
\title[Finite groups and integral cubic Cayley graphs]
{\rm On finite groups with prescribed two-generator subgroups and integral Cayley graphs}

\author[Y.-Q.~Feng]{Yan-Quan~Feng}
\address{Y.-Q.~Feng, Department of Mathematics, Beijing Jiaotong University, Beijing, \newline\indent 100044,  P.~R.~China}
\email{yqfeng@bjtu.edu.cn}

\author[I.\ Kov\'acs]{Istv\'an Kov\'acs}
\address{I.\ Kov\'acs,
UP IAM and UP FAMNIT, University of Primorska, Muzejski trg 2
\newline\indent SI-6000 Koper, Slovenia}
\email{istvan.kovacs@upr.si}

\thanks{Y.-Q.~Feng was supported by the National Natural Science Foundation of China (11731002) and the 111 Project of China (B16002). I.\ Kov\'acs was supported by the
Slovenian Research Agency (research program P1-0285 and research projects N1-0140, N1-0062, J1-9108, J1-1695 and J1-2451) and is grateful to  Beijing Jiaotong University for hospitality.}

\begin{abstract}
In this paper, we characterize the finite groups $G$ of even order with
the property that for any involution $x$ and element $y$ of $G$,
$\la x, y \ra$ is isomorphic to one of the following groups:
$\Z_2,$  $\Z_2^2$, $\Z_4$, $\Z_6$, $\Z_2 \times \Z_4$,
$\Z_2 \times\Z_6$ and $A_4$.  As a result, a characterization
will be obtained for the finite groups all of whose
Cayley graphs of degree $3$ have integral spectrum.
\end{abstract}
\maketitle
\section{Introduction}\label{sec:intro}

The class of finite groups studied in this paper was defined in a
graph theoretical context.
Let $G$ be a group with identity element $1$, and let $X \subset G$ be a subset such that $1 \notin X$ and $x^{-1} \in X$ for every $x \in X$. The {\em Cayley graph}
$\Cay(G,X)$ has vertex set $G$ and edges
$\{g,xg\}$, $g \in G$ and $x \in X$.
A graph is said to be {\em integral} if all of its eigenvalues are integers.
The study of integral graphs goes back to \cite{HS} and
the study of integral Cayley graphs was initiated in \cite{AV}.
A group $G$ is said to be {\em Cayley integral} if all Cayley graphs
$\Cay(G,X)$ are integral (see \cite{KS}). The Cayley integral groups were classified independently in \cite{AJ,ABM} and these are the following groups:
$$
D_6,~\Dic(\Z_{6}),~\Z_2^m \times \Z_3^n,~\Z_2^m \times \Z_4^n,
~\Z_2^m \times Q_8~(m,n \geqslant 0).
$$
By $\Z_n$ and $D_n$ we denote the cyclic and dihedral group
of order $n$, respectively.
Given an abelian group $A$ having a unique involution $t$ and order $|A| > 2$, the {\em generalized dicyclic group} $\Dic(A)$ is the group $\la A, x\ra$, where $x^2=t$ and $a^x=a^{-1}$ for all $a \in A$
(see \cite[page~252]{S}). The group $\Dic(\Z_4)$ is also known as the
{\em quaternion group}, denoted by $Q_8$.

Cayley integral groups were generalized in \cite{EK,MW}.
For a positive integer $k$, let $\G_k$ denote the
class of finite groups all of whose Cayley graphs of valency at most $k$ are integral (see \cite{EK}). All classes $\G_k$ with $k \neq 3$ were determined in \cite{EK}.
\smallskip

\begin{itemize}
\setlength{\itemsep}{0.6\baselineskip}
\item $\G_1$: all finite groups.
\item $\G_2$: all $D_8$- and $D_{12}$-free groups $G$ with the property that the order $$
o(x) \in \{1,2,3,4,6\}~\text{for every}~x \in G.
$$
\item $\G_4=\G_5$: the Cayley integral groups and the groups
$\Dic(\Z_3^n \times \Z_6),\, n \geqslant 1$.
\item $\G_k,\, k \geqslant 6$: the Cayley integral groups.
\end{itemize}
\smallskip

For $k \geqslant 2$, let $\A_k$ be the class of groups that admit at least one Cayley
graph of valency $k$ and
all their $k$-valent Cayley graphs are integral (see \cite{MW}).
It was shown that $\G_3$ is the union of $\A_3$ and the class of all
$3$-groups of exponent $3$ (see \cite[Corollary~2.7]{MW}).
In particular, all non-nilpotent groups in $\G_3$ are also in $\A_3$.
Ma and Wang~\cite{MW} gave the following purely group
theoretical characterization of the class $\A_3$.

\begin{thm}\label{MW}
{\rm (\cite[Theorem~2.6]{MW}).}
A finite group $G \in \A_3$ if and only if
\begin{enumerate}[{\rm (i)}]
\item $G \cong D_6$ or
\item $G$ has even order and for all $x, y \in G$ such that $o(x)=2$,
we have that $\la x, y \ra$ is isomorphic to one of the following
groups:
$$
\Z_2,~\Z_2^2,~\Z_4,~\Z_6,~\Z_2 \times \Z_4,~
\Z_2 \times\Z_6~\text{and}~A_4.
$$
\end{enumerate}
\end{thm}

By $A_4$ we denote the alternating group of degree $4$.
As a corollary of Theorem~\ref{MW}, all nilpotent groups in $\G_3$ were also found
(see \cite[Corollary~2.8]{MW}), and these are the groups:
\begin{itemize}
\item $U \times V$, where $U$ is trivial or of exponent $3$ and $V \cong \Z_2^n$, $ n \geqslant 0$.
\item The $2$-groups $G$ of exponent $4$ with
$\Omega_1(G) \leqslant Z(G)$.
\end{itemize}

Our goal in this paper is to characterize the non-nilpotent groups in class $\A_3$, or equivalently, the non-nilpotent groups having the property in case (ii) of
Theorem~\ref{MW}. We will refer to this property throughout this
article and denote it by (P).

Before presenting our result we set
some notation and definitions.
If $x, y \in G$ for some group $G$, then the {\em conjugate}
$x^y=y^{-1}xy$. A {\em semidirect product} of $G$ with another group $H$ will be denoted by $G \rtimes H$, where $G$ is normal in
$G  \rtimes H$. 
Suppose that $G$ is a transitive permutation group acting on a set
$\Omega$. Then $G$ is said to be a {\em Frobenius group} if the point
stabilizers $G_\alpha, \, \alpha \in \Omega$, are non-trivial, and
no non-identity element of $G$ fixes more than one point
(cf.\  \cite[page~37]{G}).  The elements of $G$ that fix no point in
$\Omega$ together with the identity element
form a regular normal subgroup of $G$, which is called the {\em Frobenius kernel} of $G$. The group $G$ can be expressed as $G=K \rtimes H$, where
$K$ is the Frobenius kernel of $G$ and
$H=G_\alpha$. The group $H$ is also referred to as a {\em Frobenius complement} of $G$. Note that all stabilizers $H_\alpha$ are
Frobenius complements.

\begin{rem}\label{rem:Frobenius}
Frobenius groups will arise in later proofs as subgroups of the holomorph of a group $G$. The {\em holomorph} of $G$ is
the semidirect product $G \rtimes \Aut(G)$ (see \cite[page~26]{G}).
Its elements are written uniquely of the form
$g \alpha$, where $g \in G$ and $\alpha \in \Aut(G)$ and
it has a faithful action on $G$ defined by
$x^{g\alpha}=(xg)^\alpha$, where $x, g \in G$ and $\alpha \in \Aut(G)$.
A subgroup $A \leqslant \Aut(G)$ is said to be {\em regular} if every non-identity automorphism in $A$ is {\em fixed-point-free}
(see \cite[page~39]{G}). It follows directly from the definition that,
as a permutation group of $G$ just described, $G \rtimes A$ is a  Frobenius group whose Frobenius kernel is $G$ and $A$ (the stabilizer of
the identity element of $G$) is a Frobenius complement.
\end{rem}

Finally, a $p$-group $P$ is said to be {\em special} if it is either an  elementary abelian $p$-group, or $P^\prime=\Phi(P)=Z(P)$ is an elementary abelian $p$-group (see \cite[page~183]{G}).
Note that, in the latter case, $P$ is of class $2$.
We remark that all special $p$-groups considered
in this paper will be non-abelian.

We are ready to state the main result of this paper.

\begin{thm}\label{main}
Let $G$ be a finite non-nilpotent group.
Then $G$ has property~(P)
if and only if it belongs to one of the following families:
\begin{enumerate}[{\rm (a)}]
\setlength{\itemsep}{0.5\baselineskip}
\item $\Dic(\Z_3^m \times \Z_2) \times \Z_2^n$,
$m \geqslant 1$ and $n \geqslant 0$.
\item $\big( U \times V) \rtimes \la z \ra$, where
$z$ normalizes both $U$ and $V$,
$\la U, z \ra$ is of exponent $3$,
$V=V_1 \times \cdots \times V_n$
and for every $i$, $1 \leqslant i \leqslant n$,
$$
V_i \cong \Z_2^2~\text{and}~\big\la V_i, z \big\ra \cong A_4.
$$
\item Frobenius groups $K \rtimes H$
with Frobenius kernel $K$ being a $2$-group of exponent $4$ and of class $2$ with
$\Omega_1(K) \leqslant Z(K)$ and Frobenius complement $H \cong \Z_3$.
\item $\big(U \rtimes \la z \ra\big) \times V$,
where $\big\la z \big\ra \cong \Z_3$,
$V \cong \Z_2^m$, $m \geqslant 0$,
$U$ is a non-abelian special $2$-group such that
$C_U(z)=\Omega_1(U)=Z(U)$,
$U=Z(U) \prod_{i=1}^nU_i$ and for every $i$, $1 \leqslant i \leqslant n$,
$$
Z(U)U_i \trianglelefteqslant U,~U_i \cong Q_8~\text{and}~
\big\la U_i, z \big\ra \cong \SL(2,3).
$$
Furthermore, for any $i \ne j$, $1 \leqslant i,  j \leqslant n$,
$|U_i \cap U_j| \leqslant 2$.
\end{enumerate}
\end{thm}

\begin{rem}\label{r:family-b}
Let $G$ be any $3$-group of exponent $3$ and of order $3^n$ with $n\geqslant 1$ and let $M$ be a maximal subgroup of $G$. Then $|M|=3^{n-1}$ and $M \trianglelefteqslant G$. Picking $z\in G$ but $z \not\in M$, we have $G=M\rtimes \la z \ra$. This implies that a group in family (b) can have any non-trivial $3$-group of exponent $3$ as its Sylow
$3$-subgroup.
\end{rem}

\begin{rem}\label{r:family-c}
Note that determining the groups in family (c) is equivalent to finding the $2$-groups $G$ that satisfy the following conditions:
the exponent of $G$ is $4$, $\Omega_1(G) \leqslant Z(G)$ and $G$ admits a
fixed-point-free automorphism of order $3$.
In Section~\ref{sec:5}, we will study the possible isomorphism type of a minimal non-abelian subgroup of such $2$-groups. A group is {\em minimal non-abelian} if it is
non-abelian and all of its proper subgroups are abelian. It will be shown that there are only two possibilities: the metacyclic group $H_{16}$ of order $16$ defined by
\begin{equation}\label{eq:H16}
H_{16} = \big\langle a, b \mid a^4=b^4=1, b^{-1}ab=a^{-1} \big\rangle;
\end{equation}
and the non-metacyclic group $H_{32}$ of order $32$ defined by
\begin{equation}\label{eq:H32}
H_{32} =\big\langle a, b, c \mid a^4=b^4=1, c^2=1, [a,b]=c, [a,c]=[b,c]=1 \big\rangle.
\end{equation}
Non-abelian $2$-groups with prescribed minimal non-abelian subgroups of exponent $4$ were studied by Janko~\cite{Ja}.  In Section~\ref{sec:5}, we will also discuss the
consequence of a result in \cite{Ja} when applied to the Frobenius kernels in family (c) (see Proposition~\ref{P:family-c}).
\end{rem}

\begin{rem}\label{r:family-d}
Similarly, determining the groups in family (d) is equivalent to finding the non-abelian
$2$-groups $G$ that satisfy the following conditions: $G$ is a special $2$-group,
$\Omega_1(G)=Z(G)$ and $G$ admits an automorphism of order $3$ whose fixed points comprise $Z(G)$. Examples of such $2$-groups are the Sylow
$2$-subgroups of the {\em special unitary groups} $\SU(3,q)$, $q=2^{2n+1}$ and
$n \ge 1$. Indeed, such a group admits the following matrix representation (see \cite{C}):
$$
G=\Big\{ {\small \begin{pmatrix} 1 & a & b \\ 0 & 1 & \overline{a} \\ 0 & 0 & 1
\end{pmatrix}} : a, b \in \GF(q^2),~a \overline{a}=b + \overline{b} \Big\},
$$
where $x \mapsto \overline{x}$ is the automorphism of $\GF(q^2)$ of order $2$.
This shows that 
$$
|G|=q^3=2^{6n+3}.
$$
It follows that $G^\prime=\Phi(G)=Z(G)=\Omega_1(G)$ consisting of
those matrices with $a=0$. This is an elementary abelian group of order $q=2^{2n+1}$
and we obtain that $G$ is a special group. Take now the diagonal matrix $z=\mathrm{Diag}(\lambda,1,\lambda)$ where $\lambda^3=1$ and $\lambda \ne 1$.
One can quickly check that $z$ normalizes $G$ and $C_G(z)=\Omega_1(G)$, in other
words, the action of $z$ on $G$ by conjugation induces the required automorphism
of order $3$. Another infinite family of $2$-groups with the aforementioned properties
will be exhibited in Section~\ref{sec:5} (see Proposition~\ref{P:family-d}).
\end{rem}

The paper is organized as follows: Section~\ref{sec:known} contains the preliminary
results needed in the proof of Theorem~\ref{main}.  Following several preparatory
lemmas derived in Section~\ref{sec:lemmas}, the proof is presented in Section~\ref{sec:proof}. Finally, Section~\ref{sec:5} contains some further discussion on the groups in families (c) and (d).
\section{Preliminaries}\label{sec:known}

All groups in this paper will be finite.
Our terminology and notation from group theory are standard and we follow the book \cite{G}. For an easier reading of the text we collect below some results.

To start with, we review a couple of the properties of a {\em fixed-point-free} automorphism.

\begin{prop}\label{fixed-point-free}
Let $\alpha$ be a fixed-point-free automorphism of a group $G$ such that $\alpha$ has order $n$ and let $x \in G$. Then the following conditions hold:
\begin{enumerate}[{\rm (i)}]
\item {\rm (\cite[Lemma~10.1.1(ii)]{G})}
$x (x^\alpha) \cdots (x^{\alpha^{n-1}})=1$;
\item {\rm (\cite[Theorem~10.1.4]{G})}
if $n=2$, then $G$ is abelian and $x^\alpha=x^{-1}$;
\item {\rm (\cite[Theorem~10.1.5]{G})} 
if $n=3$, then $G$ is nilpotent and $x$ commutes with $x^\alpha$.
\end{enumerate}
\end{prop}

Although it will not be used in this paper, the group $G$ in case (iii) is 
also known to have nilpotency class at most $2$ (\cite[page~90]{B}). 

Recall that, a group $A$ of automorphisms of a group $G$ is regular
if every non-identity element in $A$ is fixed-point free, and
then the semidirect product $G \rtimes A$ is a Frobenius group
(see Remark~\ref{rem:Frobenius}).
We shall use the following property of Frobenius groups.

\begin{prop}\label{Frobenius}
{\rm (cf.\ \cite[Theorem~10.3.1]{G}(iv))}
Let $G=K \rtimes H$ be a Frobenius group with kernel $K$ and complement $H$ such that $|H|$ is odd. Then every Sylow $p$-subgroup of $H$ is cyclic.
\end{prop}

\begin{prop}\label{Glemma}
{\rm (\cite[Lemma~10.5.4]{G})}
Let $G$ be a group in which $G^\prime$ is nilpotent and let $P$ be a Sylow $p$-subgroup of $G$. Then $G=O_{p^\prime}(G^\prime)N_G(P)$.
\end{prop}

Recall that, a group is minimal non-abelian if it is
non-abelian and all of its proper subgroups are abelian.

\begin{prop}\label{BJlemma}
{\rm (\cite[Lemma~3.2(a)]{BJ}).}
Let $G$ be a two-generator $p$-group such that $|G^\prime|=p$.
Then $G$ is minimal non-abelian.
\end{prop}

Let $G$ be a finite $p$-group.
Set $G^p=\la g^p\ |\ g\in G \ra$.
The following theorem is Burnside's basis theorem.

\begin{prop}\label{burnsidebasis}
{\rm (\cite[5.3.2]{Ro})}  Let $G$ be a finite
$p$-group. Then $\Phi(G)=G^\prime G^p$. If $|G:\Phi(G)|=p^r$, then every
generating set of $G$ has a subset of $r$ elements which also
generates $G$. In particular, $G/\Phi(G) \cong \Z_p^r$.
\end{prop}

The following result is Burnside's $p$-complement theorem.

\begin{prop}\label{Burnside}
{\rm (cf.\ \cite[Theor                                                                                                                                                                                                                                                                                                                                                                                                                                                                                                                                                                                                                                                                                                                                                                                                                                                                                                                                                                                                                                                                                                                                                                                                                                                                                                                                                                                        em~7.4.3]{G})}
If a Sylow $p$-subgroup $P$ of a group $G$ lies in $Z(N_G(P))$, then $G$ has a normal
$p$-complement, i.e. a normal subgroup $N$ such that $|N|=|G : P|$.
\end{prop}

R\'edei~\cite{R} classified  the minimal non-abelian $p$-groups and it follows from his result that there are exactly three minimal non-abelian $2$-groups $G$ of exponent $4$ having the property that $\Omega_1(G) \leqslant Z(G)$. Namely,
$G \cong Q_8$ or $H_{16}$ or $H_{32}$, where $H_{16}$ is the
metacyclic group of order $2^4$ defined in Eq.~\eqref{eq:H16}
and $H_{32}$ is the non-metacyclic group of order $2^5$ defined in
Eq.~\eqref{eq:H32}.

\begin{prop}\label{EK2}
{\rm (cf.\ \cite[Proposition~12]{EK} and \cite[Corollary~2.8]{MW})}
Let $G$ be a non-abelian $2$-group of exponent $4$.
The following statements are equivalent:
\begin{enumerate}[{\rm (i)}]
\item The group $G$ has property~(P).
\item $\Omega_1(G) \leqslant Z(G)$.
\item Every minimal non-abelian subgroup of $G$ is isomorphic to one of the following groups: $Q_8, \, H_{16}$ and $H_{32}$.
\end{enumerate}
\end{prop}
\section{Preparatory lemmas}\label{sec:lemmas}

Let $\pi(G)$ denote the set of all prime divisors of the order of
a group $G$ and for a prime $p \in \pi(G)$, denote by $G_p$ a
Sylow $p$-subgroup of $G$.
For elements $x, y \in G$, the {\em commutator} $[x,y]=x^{-1}y^{-1}xy$.
In our first lemma we collect several properties
of the groups with property (P), which are
direct consequences of the definition and will be used
repeatedly in the remainder of the paper.

\begin{lem}\label{L:first}
Let $G$ be a group with property (P).
Then $G$ has the following properties:
\begin{enumerate}[{\rm (i)}]
\item For each $g \in G$, the order $o(g)=1, 2, 3, 4$ or $6$,
in particular, we have $\pi(G) \subseteq \{2,3\}$ and $G$ is a soluble group.
\item If $x, y\in G$ with $o(x)=2$ and $[x,y] \ne 1$, then
$o(y)=3$ and $xx^yx^{y^2}=1$, in particular, any two involutions in $G$ commute.
\item A Sylow $2$-subgroup $G_2$ of $G$ has exponent at most $4$ and
$\Omega_1(G_2) \leqslant Z(G_2)$.
\item A Sylow $3$-subgroup $G_3$ of $G$ has exponent $3$.
\item Each subgroup $H \leqslant G$ of even order and
each quotient group $G/N$ such that $N$ is a $3$-group have
property (P).
\end{enumerate}
\end{lem}
\begin{proof}
(i): The possibilities for $o(g)$ follow directly from property (P)
and the solubility of $G$ from Burnside's $p^aq^b$-theorem.

(ii): If $x, y\in G$ with $o(x)=2$ and $[x,y] \ne 1$, then we have 
$\la x, y \ra \cong A_4$ by property (P). This implies (ii).

(iii)-(iv): These follow directly from property (P) and (ii).

(v): The first part is also clear.
For the second part, let $Nx$ be an element of $G/N$ of order $2$, 
and let $Ny$ be any other element. Note that, in this case $o(x)=2$ must hold.
Then in $G/N$ we have $\la Nx, Ny \ra \cong \la x , y \ra N/N \cong
\la x, y \ra/(\la x, y \ra \cap N)$. Furthermore, as $o(x)=2$, $\la x, y \ra$ is isomorphic to one of the groups:
$$
\Z_2,~\Z_2^2,~\Z_4,~\Z_6,~\Z_2 \times \Z_4,~
\Z_2 \times\Z_6~\text{and}~A_4.
$$
If $\la x, y\ra$ is abelian, then $\la Nx, Ny\ra$ is also isomorphic to
one of the first six groups above (here we use that $|\la Nx, Ny \ra|$
is even). Finally, if $\la x, y\ra \cong A_4$, we have that 
$\la x, y\ra \cap N$ is trivial, and $\la Nx, Ny \ra \cong A_4$.
We conclude that $G/N$ has property (P).
\end{proof}
	
We set the notation
$\Omega_1(G)=\big\la x : x \in G, \, o(x)=2 \big\ra$, 
where $G$ has even order.
Clearly, $\Omega_1(G)$ is normal in $G$.

\begin{lem}\label{L:structure}
Let $G$ be a group with property~(P).
\begin{enumerate}[{\rm (i)}]
\item $\Omega_1(G)=\Omega_1(G_2)$ is an elementary abelian $2$-group, and
\begin{equation}\label{eq:structure}
(G_2)^\prime \leqslant \Phi(G_2) \leqslant \Omega_1(G_2)
\leqslant Z(G_2).
\end{equation}
In particular, $G_2$ has nilpotency class at most $2$.
\item If $O_3(G)=1$ and $G$ has an element of order $4$, then 
$O_2(G) > \Omega_1(G_2)$.
\item If $O_2(G) > \Omega_1(G_2)$ and $G$ is not a $2$-group, then $|G_3|=3$ and the action
of $G_3$ on $O_2(G)$ by conjugation induces a regular group of automorphisms of $O_2(G)/\Omega_1(G_2)$.
\end{enumerate}
\end{lem}

\begin{proof}
(i): By Lemma~\ref{L:first}(ii), any two involutions in $G$ commute,
hence $\Omega_1(G)$ is an elementary abelian $2$-group and as
$\Omega_1(G) \trianglelefteqslant G$, we have $\Omega_1(G) \le G_2$ and so
$\Omega_1(G)=\Omega_1(G_2)$.
Also, $\Omega_1(G_2) \leqslant Z(G_2)$ by Lemma~\ref{L:first}(iii).
If $G_2$ has exponent $2$, then \eqref{eq:structure} clearly holds. If $G_2$ has exponent $4$, then, for every $x \in G_2$, we have 
$(x\Omega_1(G_2))^2=\Omega_1(G_2)$, showing that $G_2/\Omega_1(G_2)$ is an elementary abelian $2$-group.
Therefore, $(G_2)^\prime \leqslant \Phi(G_2) \leqslant \Omega_1(G_2)$ and as
$\Omega_1(G_2) \leqslant Z(G_2)$ also holds, \eqref{eq:structure} follows.
\medskip

(ii): Suppose $O_3(G)=1$ and $G$ has an
element of order $4$. By Lemma~\ref{L:first}(i),
$\pi(G) \subseteq \{2,3\}$ and $G$ is soluble. Since $O_3(G)=1$, we have the {\em Fitting subgroup} $F(G)=O_2(G)$.  By (i),
$\Omega_1(G)=\Omega_1(G_2) \leqslant O_2(G)$.  Assume 
$O_2(G)=\Omega_1(G_2)$. Then, by (i),  $G_2 \leqslant C_G(F(G)) \leqslant F(G)=\Omega_1(G_2)$ by \cite[Theorem~6.1.3]{G}.
As $G_2$ has elements of order $4$, this contradicts (i). Hence
$O_2(G) > \Omega_1(G_2)$.
\medskip

(iii): Let $g \in G_3 \setminus \{1\}$.
Then $g$ has order $3$ by Lemma~\ref{L:first}(iv), and $G_3$ normalizes $O_2(G)$.
By \cite[Theorem~5.3.15]{G},
$C_{O_2(G)/\Omega_1(G)}(g)=C_{O_2(G)}(g)\Omega_1(G)$.
As $G$ has no elements of order $12$ by Lemma~\ref{L:first}(i) we deduce that
$$
C_{O_2(G)}(g) \leqslant \Omega_1(G). 
$$
This proves (iii) if $|G_3|=3$.
If $|G_3| > 3$, then, as $G_3$ has exponent $3$, there exist $g, h \in G_3$ such
that $\langle g, h \rangle$ is elementary abelian of order $9$. But then \cite[Theorem~5.3.16]{G} implies
$$
O_2(G)=\big\langle C_{O_2(G)}(x) : x \in  \langle g, h \rangle \setminus \{1\}
\big\rangle \leqslant \Omega_1(G_2),
$$
a contradiction. Hence $|G_3|=3$.
\end{proof}

We continue with two lemmas about the normalizer of a Sylow $2$- and $3$-subgroup, respectively.

\begin{lem}\label{L:normalizer2}
Let $G$ be a group with property~(P).
Let $N=N_G(G_2)$, and let $z \in N$ be an element of order $3$. Then one of the following holds:
\begin{enumerate}[{\rm (i)}]
\item $C_{G_2}(z)=1$, and for every $u \in G_2$ of order $4$, we have that
$$
\la u, u^z \ra \cong \la u, u^z, u^{z^2} \ra \cong \Z_4^2.
$$
\item $C_{G_2}(z)=\Omega_1(G_2)$, and for every $u \in G_2$ of order $4$, we have
that
$$
\la u, u^z \ra \cong Q_8~\text{and}~\la u, u^z, u^{z^2} \ra \cong Q_8~\text{or}~Q_8 \times \Z_2.
$$
\end{enumerate}
\end{lem}

\begin{proof}
By Lemma~\ref{L:first}(i), $G$ has no element of order $12$ and this
shows that $C_{G_2}(z) \le \Omega_1(G_2)$.
We show first that 
$$
C_{G_2}(z)=1~\text{or}~C_{G_2}(z)=\Omega_1(G_2).
$$
If not, we can choose two involutions $x, y \in G$ such that $[x,z]=1$ and
$[y,z]\ne 1$. By property~(P),
$\la y,z \ra \cong A_4$. Since $[x,y]=1$ also holds, see Lemma~\ref{L:first}(ii), we have that $\la y, zx \ra \cong A_4 \times \Z_2$, a contradiction.

In order to prove the remaining statements in (i) and (ii), we assume
that $G_2$ has an element of order $4$, say $u$.
Let $K=\la G_2, z \ra$.
Since $G_2$ is normalized by $z$, it follows that $K$ can be written as the semidirect product
$K=G_2 \rtimes \la z \ra$. It is clear that $G_2$ is the unique
Sylow $2$-subgroup of $K$.
Since $K$ has no element of order $12$, see Lemma~\ref{L:first}(i), it follows that
$O_3(K)=1$.
By Lemma~\ref{L:first}(v), $K$ also has property~(P),
hence Lemma~\ref{L:structure}(ii) and (iii) can be applied to
$K$. We conclude that the action of $z$ on $G_2$ by conjugation induces a fixed-point-free automorphism of
$G_2/\Omega_1(G_2)$, say $\alpha_z$.
For sake of simplicity, for an element $v \in G_2$, we denote by
$\overline{v}$ the image
of $v$ under the canonical epimorphism $\pi_{G_2/\Omega_1(G_2)}$
from $G_2$ to $G_2/\Omega_1(G_2)$.

Let $H=\la u, u^z, u^{z^2} \ra$.
By the above paragraph, we can apply
Proposition~\ref{fixed-point-free}(i) to $\alpha_z$ and obtain that
$\overline{u}\, (\overline{u})^{\alpha_z} \, (\overline{u})^{(\alpha_z)^2}=\overline{1}$ holds in $G_2/\Omega_1(G_2)$. This means that $u u^z u^{z^2}=x$ for some 
$x \in \Omega_1(G_2)$ and thus
$H=\la u, u^z, x \ra$. By Lemma~\ref{L:structure}(i), $x \in Z(H)$. Note that $H$ is normalized by $z$ but not centralized,
and therefore, $3$ divides the order of $\Aut(H)$.
\medskip

\noindent{\bf Case~1.} $C_{G_2}(z)=1$.\ 
In this case the action of $z$ by conjugation on $G_2$ induces a fixed-point-free automorphism of $G_2$.
By Proposition~\ref{fixed-point-free}(i), $x=1$ holds above, and so $H=\la u, u^z\ra$. Also, by
Proposition~\ref{fixed-point-free}(iii), $[u,u^z]=1$, and it follows that $H \cong \Z_4 \times \Z_2$ or
$\Z_4^2$. As the former group admits no fixed-point-free automorphism
of order $3$, we conclude that $H \cong \Z_4^2$ and (i) follows.
\medskip

\noindent{\bf Case~2.} $C_{G_2}(z)=\Omega_1(G_2)$.\ 
In this case $u^2=(u^z)^2$. If $[u,u^z]=1$, then we have 
$\la u,u^z \ra \cong \Z_4 \times \Z_2$ and so
$H \cong \Z_4 \times \Z_2^r$ for $r=1$ or $2$. In either case $z$ centralizes an
element of order $4$, a contradiction.
Thus $\la u, u^z \ra$ is non-abelian.

Let $T=\la u, u^z \ra$.
By Lemma~\ref{L:structure}(i), $[u,u^z] \in Z(T)$.
Since $u^2=(u^z)^2 \in Z(T)$ also holds, we have 
$$
|\la u^2,[u,u^z]\ra|\leqslant 4~\text{and}~|T/\la u^2,[u,u^z]\ra|\leqslant 4.
$$
It follows that $|T|\leqslant 16$. Also, $T^\prime=\la [u,u^z] \ra$, and thus by Proposition~\ref{BJlemma},
$T$ is a minimal non-abelian $2$-group. We can apply Proposition~\ref{EK2} and this shows that
$T\cong Q_8$ or $H_{16}$. Recall that $H=\la T,\, x\ra$ and $x \in Z(H)$. We have $H \cong Q_8$, $Q_8 \times \Z_2$, $H_{16}$ or $H_{16} \times \Z_2$. A computation with the computer algebra system \texttt{magma}~\cite{BCP} shows that
in the latter two cases $|\Aut(H)|$ is not divisible by $3$. Thus, $T \cong Q_8$ and $H \cong Q_8$ or
$Q_8\times \Z_2$ and (ii) follows.
\end{proof}

\begin{lem}\label{L:normalizer3}
Let $G$ be a group with property~(P)
such that $G$ is not a $2$-group and let $N=N_G(G_3)$.
\begin{enumerate}[{\rm (i)}]
\item If $|N|$ is even, then
$G_3\Omega_1(N_2)=G_3 \times \Omega_1(N_2)$.
\item If $N$ has no element of order $4$, then $N \cong G_3 \times \Z_2^n$ with $n \geqslant 0$, and if $N$ has an element of order $4$, then $N \cong \Dic(\Z_3^m \times \Z_2) \times \Z_2^n$ with $m
\geqslant 1, \, n \geqslant 0$.
\item If $N$ has no element of order $4$ and $G_3$ is abelian, then
$O_2(G)=G_2$.
\end{enumerate}
In particular, if $G_3$ is non-abelian, then
$N=G_3 \times \Omega_1(N_2)$. 
\end{lem}

\begin{proof}
(i): Assume that $|N|$ is even. By Lemma~\ref{L:first}(v), $N$ also has property~(P). Thus Lemma~\ref{L:structure}(i) can be applied to $N$, and we find that
$\Omega_1(N_2) \trianglelefteqslant N$.
Since $G_3 \trianglelefteqslant N$ also holds, it follows that
$G_3\Omega_1(N_2)=G_3 \times \Omega_1(N_2)$.
\medskip

(ii): If $N$ has no element of order $4$, then $N_2=\Omega_1(N_2)$, and by (i), we have $N=G_3\times N_2$, and
$N_2  \cong \Z_2^n$ with $n \geqslant 0$.

Suppose now that $N_2$ contains an element $u$ of order $4$. By part (i), we have $[G_3,u^2]=1$. Since
no element can have order $12$, the action of $u$ on $G_3$ by conjugation induces a fixed-point-free automorphism of $G_3$ of order $2$. By Proposition~\ref{fixed-point-free}(ii), $G_3 \cong \Z_3^m$ with $m \geqslant 1$ and $z^u=z^{-1}$ for every $z \in G_3$. Similarly, $z^v=z^{-1}$ for any $v \in N_2$ with $o(v)=4$. Thus $z^{uv}=z$, hence $uv$ has order $1$ or $2$. It follows that $uv \in \Omega_1(N_2)$, implying that $|N_2 : \Omega_1(N_2)|=2$. By Lemma~\ref{L:structure}(i),
$[u,w]=1$ for any $w\in \Omega_1(N_2)$, and it follows $N_2 \cong \Z_2^n \times \Z_4$ with $n \geqslant 0$.
In particular, $N \cong \Dic(\Z_3^m \times \Z_2) \times \Z_2^n$ with $m \geqslant 1$ and $n \geqslant 0$.
\medskip

(iii): This follows from (ii) and Proposition~\ref{Burnside}.
\end{proof}

In the next lemma, we address the question of whether or not a Sylow $p$-subgroup is normal.

\begin{lem}\label{L:ifnormal1}
Let $G$ be a group with
property~(P) and suppose that $G^\prime$ is nilpotent. Then $G$ has a normal Sylow $2$- or $3$-subgroup.
Furthermore, if $G_3$ is non-abelian, then $G$ has a normal Sylow $2$-subgroup.
\end{lem}

\begin{proof} Clearly, we may assume that $G$ is non-nilpotent.
In particular, $G_3 \ne 1$. Let $N=N_G(G_3)$.

Assume first that $G_3$ is non-abelian.
By Lemma~\ref{L:normalizer3}, 
$N=G_3 \times M$, where $M \cong \Z_2^n$ for some $n \geqslant 0$.  Since $G^\prime$ is nilpotent, $G$ factorizes as 
$G=O_2(G^\prime)N$ (Proposition~\ref{Glemma}).
Thus $G=O_2(G^\prime)(G_3 \times M)$, and $O_2(G^\prime)M$ is a Sylow $2$-subgroup of $G$. The group
$O_2(G^\prime)$ is characteristic in $G^\prime$, which is characteristic in $G$, implying
that $O_2(G^\prime) \trianglelefteqslant G$.
We conclude that both $O_2(G^\prime)$ and $M$ are normalized
by $G_3$ and hence that $O_2(G^\prime)M$ is normal in $G$.

Assume now that $G_3$ is abelian. 
If $N$ has no element of order $4$, then we are done by
Lemma~\ref{L:normalizer3}(iii). Assume that $N$ has elements of order $4$.
Then 
$$
N \cong \Dic(\Z_3^m \times \Z_2) \times \Z_2^n
$$ 
with $m \geqslant 1, \, n \geqslant 0$ by Lemma~\ref{L:normalizer3}(ii).
We obtain that $G_3 \leqslant N^\prime \leqslant G^\prime$.
This implies that $G_3$ is a normal Sylow $3$-subgroup of
$G^\prime$ because $G^\prime$ is nilpotent, and hence $G_3$ is also normal in $G$.
\end{proof}

\begin{lem}\label{L:ifnormal2}
Let $G$ be a group with property~(P) such that $G$ is not a $2$-group,
$G_2$ is of exponent $4$ and $G_3 \ntrianglelefteqslant G$.
Then
$$
|G_3|=3~\text{and}~|G_2 : O_2(G)| \leqslant 2.
$$
Furthermore, if $|G_2 : O_2(G)|=2$, then $G/O_2(G) \cong D_6$.
\end{lem}

\begin{proof}
Assume for the moment that $O_3(G) \ne 1$.
We are going to use below the facts that both $O_3(G)G_2$ and $G/O_3(G)$ have property~(P), see Lemma~\ref{L:first}(v).

Since $O_3(G)$ is normal in $O_3(G)G_2$ and $G_2$ is
of exponent $4$, we obtain by Lemma~\ref{L:normalizer3}(ii) that
$G_2 \cong \Z_2^n \times \Z_4$ for some $n \geqslant 0$, and that
$O_3(G) \cong \Z_3^m$ for some $m \geqslant 1$.
Since $G_3 \ntrianglelefteqslant G$, it follows that $|G_3| \ne 3$.
This, combined together with Lemma~\ref{L:structure}(iii), yields that
$$
O_2(G)=\Omega_1(G_2) \cong \Z_2^{n+1},
$$ 
and so $|G_2 :O_2(G)|=2$.
Note that, $G_p \ntrianglelefteqslant G$ for both $p=2$ and $p=3$.

Write $H=G/O_3(G)$. Then $O_3(H)=1$ and
$H_2 \cong G_2  \cong \Z_2^n \times \Z_4$.
Consequently, $|H_2 : \Omega_1(H_2)|=2$.
By Lemma~\ref{L:structure}(ii), $O_2(H) \ne \Omega_1(H_2)$, implying that $O_2(H)=H_2$, and so $G_2O_3(G) \trianglelefteqslant G$.
Also, by Lemma~\ref{L:structure}(iii), $|H_3|=3$.

The Fitting subgroup 
$$
F(G)=O_2(G) \times O_3(G)
$$ 
is abelian. The order $|G/F(G)|=6$ because $|G_p: O_p(G)|=p$ for both $p=2$ and $p=3$,
hence $G/F(G) \cong \Z_6$ or $D_6$. The latter case is impossible
because $G_2F(G)/F(G)=G_2O_3(G)F(G)/F(G)$ is a normal subgroup of $G/F(G)$ of order $2$.
We conclude in turn that, $G/F(G) \cong \Z_6$, $G^\prime \leqslant F(G)$ and hence
$G^\prime$ is abelian. By Lemma~\ref{L:ifnormal1},
$G_2 \trianglelefteqslant G$ or $G_3 \trianglelefteqslant G$, a contradiction.

Summing up, we have proved that $O_3(G)=1$. As $G_2$ is of exponent $4$, by Lemma~\ref{L:structure}(ii), $O_2(G) > \Omega_1(G_2)$,
and then by Lemma~\ref{L:structure}(iii), $|G_3|=3$.

Assume that $G_2 \trianglelefteqslant G$. Then $G_2=O_2(G)$ and the lemma follows.

Now, assume that $G_2 \ntrianglelefteqslant G$.
Let us consider the groups $K:=G/O_2(G)$ and $L:=G_3O_2(G)/O_2(G)$. Then
$O_2(K)=1$ and $|L|=3$, implying in turn that
$L \trianglelefteqslant K$, $C_K(L)=L,$ and $K/L$ is isomorphic to a
subgroup of $\Aut(L)\cong \Z_2$. Since $G_2 \ntrianglelefteqslant G$, we conclude that
$|K|=2|L|$ and $K  \cong D_6$. This completes the proof of the lemma.
\end{proof}
\section{Proof of Theorem~\ref{main}}\label{sec:proof}

Let $K \rtimes H$ be a Frobenius group with Frobenius kernel
$K \cong \Z_2^n$ and a complement $H \cong \Z_3$.
Let $H=\la z \ra$, and let $x \in K$ be any
non-identity element. The action of $z$ on $K$ by conjugation induces a fixed-point-free automorphism. Then Proposition~\ref{fixed-point-free}(i) implies that
$\la x, z \ra \cong A_4$. This fact will
be used several times in the rest of the section.

\subsection{Sufficiency}
Let $G$ be a group from one of the 
families (a)--(d) in Theorem~\ref{main}.
We have to show that $G$ has property~(P).

It is straightforward to check that for any $g \in G$,
$o(g)=1, 2, 3, 4$ or $6$.
Let $x, y \in G$ be arbitrary elements such that $o(x)=2$.
In the case when $[x,y]=1$, we find
$$
\la x, y \ra \cong \Z_2,~\Z_2^2,~\Z_4,~\Z_6,~\Z_2 \times \Z_4~\text{or}~\Z_2 \times \Z_6.
$$
Thus $G$ has property~(P) if and only if the
following condition holds:
\begin{equation}\label{eq:suffi}
[x,y] \ne 1 \implies \la x, y \ra \cong A_4.
\end{equation}

If $G$ is from family (a), then one can check easily that
$x \in Z(G)$, in particular,
the above implication holds trivially.

Suppose that $G=(U \times V) \rtimes \la z \ra$ given in family (b).
Then we have $x \in V$ and $y=uvz^i$, where $u\in U,\, v\in V$ and
$0 \leqslant  i \leqslant 2$. g Writing $t=u z^i$, we have that $y=vt$ and
$y^3=v v^{t^2} v^t=v v^t v^{t^2}=v v^{z^i} v^{z^{2i}}$.
It follows from the description of the groups in family (b) that $z$ acts on
$V$ by conjugation as a fixed-point-free automorphism of order $3$, and so
$y^3=1$ if $i > 0$ by Proposition~\ref{fixed-point-free}(i). If $[x,y] \ne 1$, then $i > 0$, $y$ normalizes $V$, and its action on $V$ by conjugation induces a fixed-point-free automorphism of order $3$. By the paragraph preceding Subsection~4.1,
$\la x, y \ra \cong A_4$.

Suppose next that $G=K \rtimes H$ given in family (c).
Then we have $x \in \Omega_1(K)$ and $y=kh$, where $k\in K$ and $h\in H$.
If $h=1$, then we have $[x,y]=1$ because $\Omega_1(K) \leqslant Z(K)$. If
$h \ne 1$, then $y$ has order $3$ and $\la x,y \ra \cong A_4$
because $y$ normalizes
$\Omega_1(K)$ and its action on $\Omega_1(K)$ by conjugation induces a fixed-point-free automorphism of order $3$. Thus the paragraph preceding Subsection~4.1 can be applied again and this shows that \eqref{eq:suffi} holds also in this case.

Finally, suppose that $G=(U \rtimes \la z \ra) \times V$ given in family (d). It follows from the properties of groups in family (d) that
$\Omega_1(G)=\Omega_1(U) \times V=Z(G)$. This yields \eqref{eq:suffi} immediately.

\subsection{Necessity}
Let $G$ be a group having property~(P) such that $G$ is non-nilpotent.
Note that, $\Omega_1(G_2) \leqslant Z(G_2)$ holds by Lemma~\eqref{L:first}(iii).
We are going to show below 
that $G$ belongs to one of the families (a)--(d) in Theorem~\ref{main}.
\medskip

\noindent{\bf Case 1.} $G_3 \trianglelefteqslant G$.\ 
We apply Lemma~\ref{L:normalizer3} to $G$. If $G_2$ is an elementary abelian
$2$-group, then $G_2=\Omega_1(G_2)$, hence $G=G_3 \times G_2$, see Lemma~\ref{L:normalizer3}(i),
contradicting that $G$ is non-nilpotent.
Therefore, $G$ is of exponent $4$, and Lemma~\ref{L:normalizer3}(ii) shows that
it belongs to family (a) in Theorem~\ref{main}.
\medskip

\noindent{\bf Case 2.} $G_3 \ntrianglelefteqslant G$,
$G_2 \trianglelefteqslant G$
and $G_2$ is an elementary abelian $2$-group.\ Let 
$$
U=C_{G_3}(G_2)~\text{and}~V=G_2.
$$
Lemma~\ref{L:normalizer2} implies that for every $z \in G_3$ such that $z \notin U$, the action of $z$
on $V$ by conjugation induces a fixed-point-free automorphism,
and such a $z$ exists because $G$ is non-nilpotent. Consequently,
the group $G_3/U$ is isomorphic to a regular group $H$ of automorphism of $V$, and
the semidirect product $V\rtimes H$ is a Frobenius group with kernel $V$ and a complement $H$.
By Proposition~\ref{Frobenius}, $|H|=3$. Therefore, $|G_3 : U|=3$, and if $z \in G_3$ such that $z \notin U$, 
we can write $G_3=U \rtimes \la z \ra$ and so $G=(U \times V) \rtimes \la z \ra$.

Let us consider $V$ as a $\la z \ra$-module over $\GF(2)$.
Since $G$ is non-nilpotent, by Lemma~\ref{L:normalizer2} we have $C_V(z)=1$. Thus $V$ can be decomposed into the direct sum $V=V_1 \oplus \cdots \oplus V_n$ of
irreducible $\la z \ra$-modules of dimension $2$ each, and for every
$i,\, 1 \leqslant i \leqslant n$,
there exist elements $a_i, b_i \in V_i$ such that $V_i=\la a_i,b_i \ra$, $a_i^z=b_i$ and
$b_i^z=a_ib_i$. We obtain that $G$  belongs to family (b) of Theorem~\ref{main}.
\medskip

\noindent{\bf Case 3.} $G_3 \ntrianglelefteqslant G$,
$G_2 \trianglelefteqslant G$
and $G_2$ is of exponent $4$.\ By Lemma~\ref{L:ifnormal2}, $|G_3|=3$. By Lemma~\ref{L:normalizer2},
$C_{G_2}(G_3)=1$ or
$C_{G_2}(G_3)=\Omega_1(G_2)$.
In the former case we obtain that $G=G_2 \rtimes G_3$ is a Frobenius
group as described in family (c) in Theorem~\ref{main}.
\smallskip

It remains to consider the case when
$C_{G_2}(G_3)=\Omega_1(G_2)$.
Let $z$ be a generator of $G_3$.
By Lemma~\eqref{L:structure}(i), $\Omega_1(G_2) \leqslant Z(G_2)$. On the other hand, if $u \in G_2$ is an element of order $4$, then $[u^z,u]\ne 1$, see Lemma~\ref{L:normalizer2}(ii),
and we conclude that $\Omega_1(G_2)=Z(G_2)$.
Let $x \in (G_2)^2, \, x \ne 1$. Then $x=u^2$ for some $u \in G_2$ and Lemma~\ref{L:normalizer2}(ii)
implies that $x=[u,u^z]$. We find in turn that, $(G_2)^2 \leqslant (G_2)^\prime$ and $\Phi(G_2)=(G_2)^\prime (G_2)^2=(G_2)^\prime$. Then \eqref{eq:structure} reduces to
$$
(G_2)^\prime=\Phi(G_2) \leqslant Z(G_2)=\Omega_1(G_2).
$$
Let $V$ be a complement of $\Phi(G_2)$ in $Z(G_2)$.
Then, by \cite[Theorem~5.2.3]{G},
\begin{eqnarray*}
G_2/\Phi(G_2) &=& C_{G_2/\Phi(G_2)}(G_3) \times [G_2,G_3] \Phi(G_2)/\Phi(G_2) \\
&=&  V\Phi(G_2)/\Phi(G_2) \times [G_2,G_3] \Phi(G_2)/\Phi(G_2).
\end{eqnarray*}
Hence $U=[G_2,G_3]\Phi(G_2)$ is a complement of $V$ and also $G_3$-invariant, i.e.
normalized by $z$.
Therefore,
$$
G=(U \times V) \rtimes \la z \ra=(U \rtimes \la z \ra) \times V,
$$
where $V \cong \Z_2^m, \, m \geqslant 0$.
Since $G_2=U \times V$ and $V \leqslant \Omega_1(G_2)=Z(G_2)$, we have 
$U^\prime=G^\prime=\Phi(G_2)=\Phi(U)$. Also,
$Z(G_2)=Z(U) \times V$, by which 
$$
|Z(U)|=|Z(G_2) : V|=|\Phi(G_2)|,
$$
and as $Z(U) \geqslant \Phi(G_2)$, we conclude that $U^\prime=\Phi(U)=Z(U)$ and
$U$ is a non-abelian special $2$-group.

Let $K=U/Z(U)=U/\Phi(U)$, $L=\la U, z \ra=U \rtimes \la z \ra$. For $x \in L$, let
$\overline{x}$ denote the image of
$x$ under the canonical epimorphism
$\pi_{L/Z(U)}$ from $L$ to $L/Z(U)$.

By Lemma~\ref{L:structure}(i),
$\Phi(U) \leqslant \Omega_1(U) \leqslant Z(U)$.
On the other hand, $\Phi(U)=Z(U)$ also holds, hence
$\Omega_1(U)=Z(U)$. Then by Lemma~\ref{L:structure}(iii),
we have $C_K(\overline{z})=1$.
Again, considering $K$ as a $\la \overline{z} \ra$-module over
$\GF(2)$, we obtain that
$K$ decomposes into the direct sum $K=K_1 \oplus \cdots \oplus K_n$ of irreducible $\la \overline{z} \ra$-modules of dimension $2$ each.

Next, we construct $n$ subgroups $U_1, \ldots, U_n \leqslant U$.
Fix $1 \leqslant i  \leqslant n$. Choose an element $u_i$ such that
$\overline{u}_i \in K_i$ and $\overline{u}_i \ne \overline{1}$.
Then $o(u_i)=4$ and by Lemma~\ref{L:normalizer2}(ii), $\la u_i, u_i^z, u_i^{z^2} \ra \cong Q_8$ or $Q_8 \rtimes \Z_2$.
A computation with the computer algebra system \texttt{magma}~\cite{BCP} shows that the group
$Q_8 \times \Z_2$ contains four subgroups isomorphic to $Q_8$. Consequently, we can choose a
subgroup $U_i \leqslant \la u_i, u_i^z, u_i^{z^2} \ra$ such that $U_i \cong Q_8$ and $U_i^z=U_i$. This implies that $\la U_i, z \ra \cong \SL(2,3)$.
It is clear that $\pi_{L/Z(U)}(U_i)=K_i$.

We claim that the above subgroups $U_1, \ldots, U_n$ satisfy all the conditions given by family (d) of Theorem~\ref{main}, and therefore,
$G$ belongs to this family (recall that
$G=(U \rtimes \la z \ra) \times V$).

Let $1 \leqslant i, j \leqslant n$ such that $i \ne j$.
Then 
$$
U_iZ(U)/Z(U)=K_i~\text{and}~U_jZ(U)/Z(U)=K_j.
$$
Since $K_i \trianglelefteqslant K=U/Z(U)$, we have $U_iZ(U) \trianglelefteqslant U$.
Also, $K=K_1 \oplus \cdots \oplus K_n$, implying that
$U=Z(U)\prod_{i=1}^nU_i$. Finally, since $K_i \cap K_j=1$, we
obtain that $U_i \cap U_j \leqslant Z(U)$ and as $Z(U)$ is an
elementary abelian $2$-group, it follows that $|U_i \cap U_j| \leqslant 2$.
This completes the proof of Case~3.
\medskip
	
\noindent{\bf Case 4.} $G_p \ntrianglelefteqslant G$ for both $p=2$ and $3$.\ 
In fact, we are going to prove below that this case does not occur.

For sake of simplicity, we set $U=O_2(G)$.
Note that, $G_2$ is not an elementary abelian $2$-group. For
otherwise,
$G_2=\Omega_1(G_2)$, in particular, $G_2 \trianglelefteqslant G$
by Lemma~\ref{L:structure}(i), which is impossible.
Then, by Lemma~\ref{L:ifnormal2}, $|G_3|=3$, $|G_2 :U|=2$ and
$G/U \cong D_6$.

Let $z$ be a generator of $G_3$ and
$N=N_G(G_3)$.  According to Lemma~\ref{L:normalizer3} one of
the following two possibilities holds:
$N=G_3 \times N_2$ and $N_2 \cong \Z_2^n, \, n \geqslant 0$, or
$N \cong \Dic(\Z_6) \times \Z_2^m, \, m\geqslant 0$. In the former  case, $G_3 \leqslant Z(N)$, so $G$ has a normal $3$-complement by
Proposition~\ref{Burnside}. This, however, contradicts the assumption that
$G_2 \ntrianglelefteqslant G$. Therefore, there is an element
$w\in N$ of order $4$ such that $\la z, w\ra \cong \Dic(\Z_6)$.
Note that, we have $z^w=z^{-1}$.

Then $w \notin U$, otherwise $[z,w] \in U \cap G_3=1$,
a contradiction. On the other hand, $w^2 \in U$ because of $|G_2 : U|=2$.
It also follows that the group
$P:=U\la w \ra$ is a Sylow $2$-subgroup of $G$.

Also, the group
$H$ has property~(P) and hence Lemma~\ref{L:normalizer2} can be applied to $H$.
Since $[w^2,z]=1$, it follows by that lemma that
$C_U(z)=\Omega_1(U)$.

As $P$ is a Sylow $2$-subgroup of $G$, it follows from Lemma~\eqref{eq:structure}(i) that 
$$
P^\prime \leqslant \Omega_1(P) \leqslant Z(P). 
$$
In particular, the following
property holds:
\begin{equation}\label{eq:uw}
u^w=u[u,w]~\text{with}~
[u,w] \in  \Omega_1(P)  \leqslant Z(P)~\text{for all}~u \in U.
\end{equation}

Observe next that $U$ cannot be an elementary abelian $2$-group.
Otherwise, we have $U=\Omega_1(U)$, hence it is centralized by
$G_3$, recall that $C_U(z)=\Omega_1(U)$.
On the other hand, $G_3^w=G_3$ and
$G=\la U, w, G_3 \ra$ also holds, and all these would yield that
$G_3 \trianglelefteqslant G$, a contradiction.

Finally, we may choose an element $u \in U$ of order $4$.
Set $v=u^z$. Consider the group $\la U, z \ra$. Then
$z \in N_{\la U, z\ra}(U)$, and it has been shown in the proof of
Lemma~\ref{L:normalizer2} that $uvv^z=uu^zu^{z^2}=x$ for some $x \in \Omega_1(U)$ (see the paragraph before Case~1).
It follows 
$$
v^z=(uv)^{-1}x=uvx_1, 
$$
where $x_1=(uv)^{-2}x \in \Omega_1(U) \leqslant \Omega_1(P)$.

By Eq.~\eqref{eq:uw}, 
$$
u^w=uy_1~\text{and}~(uvx_1)^w=uvx_1y_2
$$ 
for $y_1, y_2 \in \Omega_1(P)  \leqslant Z(P)$.
Then, since $z^w=z^{-1}$, we can write
$$
(uv x_1 y_2)^{z^{-1}}=(uv x_1 y_2)^{z^w}=
(uv x_1)^{zw}=v^{z^2w}=u^{z^3w}=uy_1.
$$
Thus, $vy_1^z=(uy_1)^z=uv x_1 y_2$, and so
$$
y_1^z=u^v x_1 y_2.
$$
The elements $u^v, x_1 , y_2 \in P$,
$o(u^v)=4$, and $x_1, y_2 \in \Omega_1(P) \leqslant Z(P)$.
We obtain that the order of the product in the right side above is equal to $4$, however, $y_1^z$ is of order at most $2$, a contradiction.
This completes the proof of Case~4 and also Theorem~\ref{main}.
\section{On groups in families (c) and (d)}\label{sec:5}

\begin{lem}\label{L:noQ8}
Let $G=K \rtimes H$ be a Frobenius group from family (c) in Theorem~\ref{main}. Then $K$ contains no subgroup isomorphic to $Q_8$.
\end{lem}

\begin{proof}
Recall that, $K$ is a $2$-group of exponent $4$ with
$\Omega_1(K) \leqslant Z(K)$ and $H \cong \Z_3$.
By Lemma~\ref{L:structure}(i),
$K^\prime \leqslant \Phi(K) \leqslant \Omega_1(K)$.
Therefore, for any $x, y \in K$,
\begin{equation}\label{eq:id1}
[y,x]=[x,y]=[x^{-1},y]=[x,y^{-1}]
\end{equation}
(cf.\ \cite[Lemmas~2.2.4(iii) and 2.2.5(ii)]{G});
and for any $x, y , z \in K$, the identities
$[xy,z]=[x,z]^y[y,z]$ and $[x,yz]=[x,z][x,y]^z$
(cf.\ \cite[Theorem~2.2.1(i)-(ii)]{G}) reduce to
\begin{equation}\label{eq:id2}
[xy,z]=[x,z][y,z]~\text{and}~[x,yz]=[x,y][x,z].
\end{equation}

Now we begin proving the statement of the lemma
and assume for a contradiction that $K$ has elements
$a, b$ such that $\la a, b \ra \cong Q_8$.
Let $z$ be a generator of $H$.
For an easier notation we set $c=a^z$ and $d=b^z$.
Consider the group  $L:=\la a, b, c, d \ra \leqslant K$.
Since the action of $z$ on $K$ by conjugation induces a
fixed-point-free automorphism, we have
$c^z=(ac)^{-1}$ and $d^z=(bd)^{-1}$, see
Proposition~\ref{fixed-point-free}(i). Note that,
$L^z=L$. Also, as $C_K(z)=1$,
Lemma~\ref{L:normalizer2}(i) shows that $[a,c]=[a,a^z]=1$
and $[b,d]=[b,b^z]=1$.
Let $[a,d]=u$. Note also that, as $\la a, b \ra \cong Q_8$,
$[a,b]=a^2$, $\la c, d \ra \cong Q_8$ and $[c,d]=c^2$.
Then these and Eqs.~\eqref{eq:id1} and \eqref{eq:id2} yield
\begin{eqnarray*}
u^z&=&[c,d^z]=[c,b][c,d]=[b,c]c^2 \\
u^{z^2} &=& [d,c^z]a^2c^2=[d,a][d,c] a^2c^2=[a,d] a^2=u a^2.
\end{eqnarray*}
Apply $z$ to both sides of the last equality.
This results in $u=u^z c^2=[b,c]$. We find the commutator
subgroup $L^\prime$ as
$$
L^\prime=\big\la\, [a,b], [a,c], [a,d], [b,c], [b,d], [c,d]\, \big\ra=\la a^2, c^2, u \ra.
$$

If $u \notin \la a^2, c^2 \ra$ then
$L^\prime \cong \Z_2^3$. Clearly, $(L^\prime)^z=L^\prime$, implying that $l^z=l$ for some $l \in L$, $l \ne 1$, which is impossible.

Let $u \in \la a^2, c^2 \ra$. Since $u=u^z c^2$, we have $u \not= c^2$, and since $u^{z^2}=ua^2$, we have $u \not= a^2$. It follows that $u=a^2c^2$. Then
$\Omega_1(K) \leqslant Z(K)$ gives rise to 
\begin{eqnarray*}
(abc)^2 &=& abcabc=(ab)^2c^2(ab)^{-1}c^{-1}abc \\ 
&=& a^2c^2[ab,c]=a^2c^2[b,c]=a^2c^2u=1,
\end{eqnarray*} 
which shows that  $abc \in Z(K)$, implying $[a,b]=1$, 
contradicting that $\la a, b\ra \cong Q_8$.
This completes the proof of the lemma.
\end{proof}

In view of this lemma and Proposition~\ref{EK2},
for any group $G=K \rtimes H$ from family (c) in Theorem~\ref{main},
a minimal non-abelian subgroup of $K$ is isomorphic to $H_{16}$ or $H_{32}$. This brings us to the following result of Janko:

\begin{thm}\label{J}
{\rm (\cite[Theorem~2.6(b)]{Ja})}
Let $G$ be a non-abelian $2$-group all of whose
minimal non-abelian subgroups are isomorphic to $H_{16}$.  If $G$ is of exponent $4$,
then $G=U \times V$, where $V \cong \Z_2^n$, $n \geqslant 0$, and
$U$ is isomorphic to one of the following groups:
\begin{enumerate}[{\rm (i)}]
\item The metacyclic group $H_{16}$ defined in Eq.~\eqref{eq:H16}.
\item The minimal non-metacyclic group $H_{32}^*$ of order
$2^5$ defined by
$$
H_{32}^*=\big\la\, a, b, c \mid a^4=b^4=1, c^2=a^2,  [a,b]=1,
a^c=a^{-1}b^2,  b^c=b^{-1} \big\ra.
$$
\item The unique special group $H_{64}$ of order $2^6$ with
$Z(H_{64}) \cong \Z_2^2$, in which every maximal subgroup is isomorphic to $H_{32}^*$ given in (ii), and which is defined by
\begin{eqnarray*}
H_{64} &=& \big\la\, a, b, c, d \mid a^4=b^4=1, \, c^2=a^2b^2,
\, [a,b]=1, \, a^c=a^{-1},  \\
& & b^c=a^2b^{-1}, \, d^2=a^2, \, a^d=a^{-1}b^2, \,
b^d=b^{-1}, \, [c,d]=1 \, \big\ra.
\end{eqnarray*}
\item The semidirect product $W \rtimes \la z \ra$, where
$W \cong \Z_4^m$, $m \geqslant 2$, $z$ is of order $4$, and
$w^z=w^{-1}$ for each $w \in W$.
\end{enumerate}
\end{thm}

As an application of the theorem, we prove the following statement.

\begin{prop}\label{P:family-c}
Let $G=K \rtimes H$ be a Frobenius group from family (c) in Theorem~\ref{main} such that $K$ is non-abelian and it contains no subgroup isomorphic to $H_{32}$.
Then $G$ is the unique Frobenius group of order $192$ with
$K \cong H_{64}$.
\end{prop}

\begin{proof}
By Proposition~\ref{EK2} and Lemma~\ref{L:noQ8}, any
minimal non-abelian subgroup of $K$ is isomorphic to $H_{16}$.
Therefore, $K$ is isomorphic to one of the groups in cases (i)--(iv) in
Theorem~\ref{J}.

We compute by \texttt{magma}~\cite{BCP} that
$|\Aut(H_{16})|=2^5$ and $|\Aut(H_{32}^*)|=2^8$, and this excludes the groups in cases (i) and (ii).

Let $K \cong H_{64}$, the group in case (iii).
A computation with \texttt{magma}~\cite{BCP} shows that
$H_{64}$ admits a fixed-point-free automorphism of order $3$,
and all subgroups of $\Aut(H_{64})$ of order $3$ are conjugate in
$\Aut(H_{64})$. This implies that there is a unique Frobenius group
$G=K \rtimes H$ in family (c) with $K \cong H_{64}$.

Finally, let $K \cong (W \rtimes \la z \ra) \times V$, a group described
in case (iv). We finish the proof by excluding this possibility.
Assume on the contrary that there is a fixed-point-free
automorphism $\phi \in \Aut(K)$ of order $3$.
By Proposition~\ref{fixed-point-free}(iii), we have $[z, z^\phi]=1$.
Then $z^\phi=x z^i$ for some $x \in W \times V, \, x \ne 1$ and
$0 \leqslant i \leqslant 3$. Thus 
$$
1=[z,z^\phi]=[z,xz^i]=[z,x],
$$ 
implying that
$x$ has order $2$ and so $x \in \Omega_1(K) \leqslant Z(K)$. It follows that $z^\phi=xz$ or $xz^{-1}$ as $z^\phi$ has order $4$. By Proposition~\ref{fixed-point-free}(i),
$z z^\phi z^{\phi^2}=1$, and therefore, $x^\phi=z$ or $z^{-1}$,
both of which are impossible.
\end{proof}

\begin{rem}
Using \texttt{magma}~\cite{BCP} we have
found two more groups from family (c) in Theorem~\ref{main}.
The Frobenius group $K \rtimes H$ of order $768$, where
\begin{eqnarray*}
K &=& \big\la a, b, c, d \mid  a^4=b^4=c^4=d^4=1, a^b=a^3, c^d=c^3,  [a,c]=[b,d]=1, \\  
& & [a,d]=[b,c]=a^2c^2 \big\ra,
\end{eqnarray*} 
$H=\la z \ra$ and the action of
$z$ on $K$ by conjugation is defined by 
$$
a^z=c,~c^z=(ac)^{-1},~b^z=d~\text{and}~d^z=(bd)^{-1}.
$$
Furthermore, the Frobenius group $K \rtimes H$ of order $3072$, where
\begin{eqnarray*}
K &=& \big\la a, b, c, d, u, v \mid
a^4=b^4=c^4=d^4=u^2=v^2=1, \\
& & [a,b]=u, [c,d]=v, [a,c]=[b,d]=1, \\
& & [a,d]=[b,c]=uv, [u,c]=[u,d]=[v,a]=[v,b]=[u,v]=1 \big\ra,
\end{eqnarray*}
$H=\la z \ra$ and $a^z=c$, $c^z=(ac)^{-1}$,
$b^z=d$, $d^z=(bd)^{-1}$, $u^z=v$ and $v^z=uv$.
\end{rem}

Finally, we turn to family (d). As noted in
Remark~\ref{r:family-d},  to have a full description of this family
is equivalent to classifying the special $2$-groups $G$ such that
$\Omega_1(G)=Z(G)=\fix(\phi)$ for some $\phi \in \Aut(G)$ of
order $3$. By $\fix(\phi)$ we denote the set of elements of $G$ fixed
by $\phi$. Note that the class of these $2$-groups is closed under the direct product of groups. We say that a group is {\em indecomposable} if whenever $G=A \times B$, then  $A=1$ or $B=1$.

It was observed in Remark~\ref{r:family-d} that
the Sylow $2$-subgroups of the special unitary groups
$\SU(3,2^{2n+1})$ are examples of such $2$-groups,
which have order $2^{6n+3}$.
We conclude the paper with another infinite family of $2$-groups with
the aforementioned properties, which are indecomposable and have order
$2^{4n-1}$.

\begin{prop}\label{P:family-d}
For an integer $n\geqslant 1$, let $G$ be the group generated by elements $a_i,\, b_i, \, 1 \leqslant i \leqslant n$ and if
$n > 1$, then also by elements
$c_j, \, 1 \leqslant j \leqslant n-1$, which satisfy the following relations:

\begin{equation}\label{rel1}
a_i^4=b_i^4=c_j^2=1~(1 \leqslant i \leqslant n,\, 1 \leqslant j \leqslant n-1),
\end{equation}
\begin{equation}\label{rel2}
[a_i,a_j]=[b_i,b_j]=[a_i,c_k]=[b_i,c_k]=[c_k,c_\ell]=1~
(1 \leqslant i, j \leqslant n,\, 1 \leqslant k, \ell \leqslant n-1),
\end{equation}
\begin{equation}\label{rel3}
a_i^2=b_i^2,~a_i^{b_i}=a_i^{-1}~(1 \leqslant i \leqslant n),
\end{equation}
\begin{equation}\label{rel4}
[a_i,b_j]=\begin{cases}
c_{\min\{i,j\}} & \text{if}~|i-j|=1  \\
1 & \text{otherwise}
\end{cases}~(1 \leqslant i, j \leqslant n).
\end{equation}
Then $G$ is an indecomposable special $2$-group of order $2^{4n-1}$
such that 
$$
\Omega_1(G)=Z(G)=\fix(\phi)
$$ 
for some $\phi \in \Aut(G)$ of order $3$.
\end{prop}

\begin{proof}
We proceed by induction on $n$ and set the notation $U(n)$ for the group $G$ defined in the proposition. The group  $U(1) \cong Q_8$. It is easy to check that the proposition holds in this case.

Let $n > 1$, $V=\la a_1, b_1, c_1 \ra$, and let
$$
W=\big\la a_i, b_i, c_j : 2 \leqslant i \leqslant n,\,  2\leqslant j \leqslant  n-1
\big\ra.
$$
Then $V \cong Q_8 \times \Z_2$, $W \cong U(n-1)$ and we have
$V \cap W=1$.
It follows from the above relations in \eqref{rel1}--\eqref{rel4} that
$[V,a_i]=[V,b_i]=[V,c_j]=1$ for all $i > 2$ and
$j \geqslant 2$, and $V^{a_2}=V^{b_2}=V$. Thus
$V \trianglelefteqslant G$, and
therefore, $G=V \rtimes W$.
In particular, $|G|=16 \cdot |U(n-1)|$ and $|G|=2^{4n-1}$ follows by the induction hypothesis.

Let
$$
Z=\big\la a_i^2,c_j :  1 \leqslant i \leqslant n,\, 1 \leqslant j \leqslant  n-1 \big\ra.
$$
Note that $Z \cong \Z_2^{2n-1}$, $G/Z \cong \Z_2^{2n}$
and $Z \leqslant Z(G)$. Then $Z=Z_1\times Z_2$
where $Z_1=\la a_1^2\ra \times \cdots \times \la a_n^2\ra$ and
$Z_2=\la c_1 \ra \times \cdots \times \la c_{n-1} \ra$. 
Relations \eqref{rel1}--\eqref{rel4} yield that any element
$x \in G$ can be written in the form
\begin{equation}\label{eq:x}
x=x_1 \cdots x_{n+1},~\text{where}~x_i \in
\begin{cases}
\{1,a_i,b_i,a_ib_i\} & \mbox{if}~1 \leqslant i \leqslant n, \\
Z & \mbox{if}~i=n+1.
\end{cases}
\end{equation}
Observe that, since the number of products in the right side above is at most $4^{n} \cdot 2^{2n-1}=|G|$,
it follows that the above decomposition of $x$ is unique.

Given any $y \in G$,
in the rest of the proof we write $y=y_1y_2\cdots y_ny_{n+1}$ with $y_i\in \{1,a_i,b_i,a_ib_i\}$ for $1\leqslant i\leqslant n$ and $y_{n+1}\in Z$.

From relations in \eqref{rel1}--\eqref{rel4} it is easy to deduce that $Z \leqslant G^\prime$.
Since $G/Z \cong \Z_2^{2n}$ is abelian, $G^\prime \leqslant Z$ and hence $G^\prime=Z$.

Since $G^\prime=Z \leqslant Z(G)$, the identities
$$
[xy,z]=[x,z]^y[y,z]~\text{and}~[x,yz]=[x,z][x,y]^z
$$
(cf.\ \cite[Theorem~2.2.1(i)-(ii)]{G})) generalize in $G$ to 
$$
[xy,z]=[x,z][y,z]~\text{and}~[x,yz]=[x,y][x,z].
$$ 
It follows that
\begin{equation}\label{eq:id11}
[x,y]=\prod_{1\leqslant i,j \leqslant n,\ |i-j|\leqslant 1}[x_i,y_j],
\end{equation}
where $[x_i,y_i]\in Z_1$ and $[x_i,y_j]\in Z_2$ for $|i-j|=1$. Furthermore,
\begin{equation}\label{eq:id22}
x y = (x_1y_1)(x_2 y_2) \cdots (x_n y_n) z~\text{for some}~z \in Z_2.
\end{equation}

Clearly, $Z\leqslant \Omega_1(G)$. Let $x \in G$ with $x^2=1$.  By Eq.~\eqref{eq:id22}, $1=x_1^2x_2^2\cdots
x_n^2c$ for some $c\in Z_2$. Since $x_i^2\in \la a_i^2\ra$ and
 $Z_1=\la a_1^2\ra \times \cdots \times \la a_n^2\ra$ with $Z_1\cap Z_2=1$, we have  that $x_i^2=1$ for $1\leqslant i\leqslant n$, forcing $x_i=1$ because $x_i\in \{1,a_i,b_i,a_ib_i\}$. Thus $x=x_{n+1}\in Z$ and hence $\Omega_1(G)\leqslant Z$. It follows that $Z=\Omega_1(G)$.

Let $x\in Z(G)$. If $x\not\in Z$, there exists $x_r\ne 1$ with $x_i=1$ for $1\leqslant i\leqslant r-1$. Pick $z_r\in \{a_r,b_r,a_rb_r\}$ with $z_r \ne x_r$. Then $[x_r,z_r]=a_r^2 \ne 1$. Taking $y=z_r$ in Eq.~\eqref{eq:id11}, we have $[x,z_r]=[x_r,z_r][x_{r+1},z_r]$, where $1\ne [x_r,z_r]\in Z_1$ and $[x_{r+1},z_r]\in Z_2$. Since $Z_1\cap Z_2=1$, we have $[x,z_r]\ne 1$ and hence $x \not\in Z(G)$, a contradiction. Thus $Z(G)\leqslant Z$ and since $Z\leqslant Z(G)$ also holds, we have $Z(G)=Z$.

Proposition~\ref{burnsidebasis} gives rise to
$\Phi(G)=G^\prime G^2$. By Eq.~\eqref{eq:id22},
$G^2 \leqslant Z=G^\prime$, hence $\Phi(G)=G^\prime$.
Up to now, we have proved that 
$$
\Phi(G)=G^\prime=\Omega_1(G)=Z(G)=Z.
$$

We prove next that $G$ is also indecomposable.
This is clear for $G=U(1) \cong Q_8$, hence we assume that $n > 1$.

Suppose that $G=H \times K$ holds for subgroups $H$ and $K$.
Define the sets
$$
H_i=\{ x_i : x \in H \}~\text{and}~K_i=\{ y_i : y \in K\},~\text{where}~1\leqslant i\leqslant n.
$$
We may assume without loss of generality that $H$ is chosen so that
$H_1 \ne 1$. We are going to prove below
that then $K=1$, which shows that $G$ is indeed indecomposable.
In fact, it is sufficient to show that $K \leqslant Z=Z(G)$.
Indeed, then $K \leqslant \Phi(G)$, hence $G=\la H, K \ra=H$ and
so $K=1$.

Let $x \in H$ with $x_1 \ne 1$ and let $y \in K$.
By Eq.~\eqref{eq:id11}, $[x,y]=[x_1,y_1]\, t$ for some $t \in \big\la a_{i+1}^2, c_i :  1 \leqslant i \leqslant n-1 \big\ra$.
As $[x,y]=1$, we have 
$$
t=[x_1,y_1]\in \langle a_1^2\rangle\cap \la a_{i+1}^2, c_i :  1 \leqslant i \leqslant n-1 \big\ra=1,
$$
hence $[H_1,K_1]=1$. This implies that, if $H_1=\{1,x_1\}$ then $K_1 \subset \{1,x_1\}$, contradicting that $G=H \times K$. It follows that
$H_1=\{1,a_1,b_1,a_1b_1\}$ and $K_1=1$.

Since $K_1=1$, we have $y_1=1$ and we can write $1=[x,y]=[x_1,y_2]\, t$ for some
$t \in \big\la a_{i+1}^2, c_j : 1 \leqslant i \leqslant n-1, 2 \leqslant j
\leqslant n-1 \big\ra$,
implying that $[x_1,y_2]=1$. Thus $[H_1,K_2]=1$ and then
$H_1=\{1, a_1, b_1, a_1b_1\}$ yields that $K_2=1$.

Since $G=H \times K$ and $K_2=1$, we have
$H_2=\{1,a_2,b_2,a_2b_2\}$ and then an argument similar to 
the above paragraph gives rise to $K_3=1$. Continuing in the same way, we obtain $K_1= \cdots =K_n=1$, i.e. $K\leqslant Z$. This completes the proof that $G$ is indecomposable.

Finally, it remains to show that $G$ has an automorphism of order $3$
whose fixed points comprise $Z(G)$.  Define the elements $\alpha_i$,
$\beta_i$, $\gamma_j \in G$ as
$$
\alpha_i=b_i, \; \beta_i=b_ia_i^{-1}~\text{and}~
\gamma_j=c_j~(1 \leqslant i \leqslant n,\, 1 \leqslant j \leqslant  n-1).
$$
Straightforward computations yield that $\alpha_i$, $\beta_i$ and
$\gamma_j$ generate $G$ and satisfy 
relations \eqref{rel1}--\eqref{rel4} as do $a_i$, $b_i$ and $c_i$. By \cite[Proposition~4.3 and Remark~4.4]{Jo}, there exists an automorphism $\phi \in \Aut(G)$ such that
$a_i^\phi=\alpha_i$, $b_i^\phi=\beta_i$ and $c_j^\phi=\gamma_j$. It is easy to see that $\phi$ has
order $3$ and for $x \in G$, we have $x^\phi=x$ if and only if $x \in \Omega_1(G)=Z(G)$. This completes the proof of the proposition.
\end{proof}
\section*{Acknowledgements}
The authors thank the reviewer for the numerous helpful suggestions, which improved the presentation considerably. They are also grateful to Professor Michael Giudici for mentioning the Sylow $2$-subgroups of $\SU(3,2^{2n+1})$ 
and to Professor Chris Parker for his suggestions to simplify several proofs.

\end{document}